\documentclass[10pt]{amsart}

\usepackage[latin1]{inputenc}
\usepackage{amsmath}
\usepackage{amsfonts}
\usepackage{amssymb}
\usepackage[mathscr]{eucal}
\usepackage{diagrams}
\usepackage{xy,color,graphicx}
\usepackage{hyperref}
\xyoption{all}

\newcommand{\wis}[1]{{\text{\em \usefont{OT1}{cmtt}{m}{n} #1}}}

\newcommand{\C}{\mathbb{C}}

\newcommand{\vtx}[1]{*+[o][F-]{\scriptscriptstyle #1}}

\newtheorem{proposition}{Proposition}

\title[Brauer-Severi motives and DT-invariants of quantized $3$-folds]{Brauer-Severi motives and Donaldson-Thomas invariants of quantized threefolds}
\author{Lieven Le Bruyn} 
\address{Department of Mathematics, University of Antwerp \\ 
 Middelheimlaan 1, B-2020 Antwerp (Belgium) \\ {\tt lieven.lebruyn@uantwerpen.be}}

\begin{document}
\sloppy

\maketitle

\begin{abstract}
Motives of Brauer-Severi schemes of Cayley-smooth algebras associated to homogeneous superpotentials are used to compute inductively the motivic Donaldson-Thomas invariants of the corresponding Jacobian algebras. This approach can be used to test the conjectural exponential expressions for these invariants, proposed in \cite{Cazz}. As an example we confirm the second term of the conjectured expression for the motivic series of the homogenized Weyl algebra.
\end{abstract}

\section{Introduction}
We fix a homogeneous degree $d$ superpotential $W$ in $m$ non-commuting variables $X_1,\hdots,X_m$. For every dimension $n \geq 1$, $W$ defines a regular functions, sometimes called the Chern-Simons functional
\[
Tr(W)~:~\mathbb{M}_{m,n} = \underbrace{M_n(\C) \oplus \hdots \oplus M_n(\C)}_m \rTo \C \]
obtained by replacing in $W$ each occurrence of $X_i$ by the $n \times n$ matrix n the $i$-th component, and taking traces.

We are interested in the (naive, equivariant) motives of the fibers of this functional which we denote by
\[
\mathbb{M}_{m,n}^W(\lambda) = Tr(W)^{-1}(\lambda). \]
Recall that to each isomorphism class of a complex variety $X$ (equipped with a good action of a finite group of roots of unity) we associate its naive equivariant  motive $[X]$ which is an element in the ring $K_0^{\hat{\mu}}(\mathrm{Var}_{\C})[ \mathbb{L}^{-1/2}]$ (see \cite{Davison} or \cite{Cazz}) and is subject to the scissor- and product-relations
\[
[X]-[Z]=[X-Z] \quad \text{and} \quad [X].[Y]=[X \times Y] \]
whenever $Z$ is a Zariski closed subvariety of $X$. A special element is the Lefschetz motive $\mathbb{L}=[ \mathbb{A}^1_{\C}, id]$ and we recall from \cite[Lemma 4.1]{Morrison} that $[GL_n]=\prod_{k=0}^{n-1}(\mathbb{L}^n-\mathbb{L}^k)$ and from \cite[2.2]{Cazz} that $[\mathbb{A}^n,\mu_k]=\mathbb{L}^n$ for a linear action of $\mu_k$ on $\mathbb{A}^n$. This ring is equipped with a plethystic exponential $\wis{Exp}$, see for example \cite{Bryan} and \cite{Davison}.

The representation theoretic interest of the degeneracy locus $Z = \{ d Tr(W)=0 \}$ of the Chern-Simons functional is that it coincides with the scheme of $n$-dimensional representations
\[
Z = \wis{rep}_n(R_W) \quad \text{where} \quad R_W = \frac{\C \langle X_1,\hdots,X_m \rangle}{(\partial_{X_i}(W) : 1 \leq i \leq m)} \]
of the corresponding Jacobi algebra $R_W$ where $\partial_{X_i}$ is the cyclic derivative with respect to $X_i$.
As $W$ is homogeneous it follows from \cite[Thm. 1.3]{Davison} (or \cite{Behrend} if the superpotential allows 'a cut') that its virtual motive is equal to
\[
[ \wis{rep}_n(R_W) ]_{virt} = \mathbb{L}^{-\frac{mn^2}{2}}([\mathbb{M}_{m,n}^W(0)]-[\mathbb{M}_{m,n}^W(1)]) \]
where $\hat{\mu}$ acts via $\mu_d$ on $\mathbb{M}^W_{m,n}(1)$ and trivially on $\mathbb{M}^W_{m,n}(0)$.
These virtual motives can be packaged together into the motivic Donaldson-Thomas series
\[
U_W(t) = \sum_{n=0}^{\infty} \mathbb{L}^{- \frac{(m-1)n^2}{2}} \frac{[\mathbb{M}_{m,n}^W(0)]-[\mathbb{M}_{m,n}^W(1)]}{[GL_n]} t^n \]
In \cite{Cazz} A. Cazzaniga, A. Morrison, B. Pym and B. Szendr\"oi conjecture that this generating series has an exponential expression involving simple rational functions of virtual motives determined by representation theoretic information of the Jacobi algebra $R_W$
\[
U_W(t) \overset{?}{=} \wis{Exp}(- \sum_{i=1}^k \frac{M_i}{\mathbb{L}^{1/2}-\mathbb{L}^{-1/2}} \frac{t^{m_i}}{1-t^{m_i}}) \]
where $m_1=1,\hdots,m_k$ are the dimensions of simple representations of $R_W$ and $M_i \in \mathcal{M}_{\C}$ are motivic expressions without denominators, with $M_1$ the virtual motive of the scheme parametrizing (simple) $1$-dimensional representations. Evidence for this conjecture comes from cases where the superpotential admits a cut and hence one can use dimensional reduction, introduced by A. Morrison in \cite{Morrison}, as in the case of quantum affine three-space \cite{Cazz}.

The purpose of this paper is to introduce an inductive procedure to test the conjectural exponential expressions given in \cite{Cazz} in other interesting cases such as the homogenized Weyl algebra and elliptic Sklyanin algebras. To this end we introduce the following quotient of the free necklace algebra on $m$ variables
\[
\mathbb{T}_m^W(\lambda) = \frac{\C \langle X_1, \hdots, X_m \rangle \otimes \wis{Sym}(V_m)}{(W-\lambda)},~\text{where}~V_m = \frac{\C \langle X_1,\hdots,X_m \rangle}{[\C \langle X_1,\hdots,X_m \rangle,\C \langle X_1,\hdots,X_m \rangle]_{vect}} \]
is the  vectorspace space having as a basis all cyclic words in $X_1,\hdots,X_m$. Note that any superpotential is an element of $\wis{Sym}(V_m)$. Substituting each $X_k$ by a generic $n \times n$ matrix and each cyclic word by the corresponding trace we obtain a quotient of the trace ring of $m$ generic $n \times n$ matrices
\[
\mathbb{T}_{m,n}^W(\lambda) = \frac{\mathbb{T}_{m,n}}{(Tr(W)-\lambda)} \quad \text{with} \quad \mathbb{M}_{m,n}^W(\lambda) = \wis{trep}_n(\mathbb{T}_{m,n}^W) \]
such that its scheme of trace preserving $n$-dimensional representations is isomorphic to the fiber $\mathbb{M}_{m,n}^W(\lambda)$. We will see that if $\lambda \not= 0$ the algebra $\mathbb{T}_{m,n}^W(\lambda)$ shares many ringtheoretic properties of trace rings of generic matrices, in particular it is a Cayley-smooth algebra, see \cite{LBbook}. As such one might hope to describe $\mathbb{M}_{m,n}^W(\lambda)$ using the Luna stratification of the quotient and its fibers in terms of marked quiver settings given in \cite{LBbook}. However, all this is with respect to the \'etale topology and hence useless in computing motives.

For this reason we consider the Brauer-Severi scheme of $\mathbb{T}_{m,n}^W(\lambda)$, as introduced by M. Van den Bergh in \cite{VdBBS} and further investigated by M. Reineke in \cite{ReinekeBS}, which are quotients of a principal $GL_n$-bundles and hence behave well with respect to motives. More precisely, the Brauer-Severi scheme of $\mathbb{T}_{m,n}^W(\lambda)$ is defined as
\[
\wis{BS}_{m,n}^W(\lambda) = \{ (v,\phi) \in \C^n \times \wis{trep}_n(\mathbb{T}_{m,n}^W(\lambda)~|~\phi(\mathbb{T}_{m,n}^W(\lambda))v = \C^n \} / GL_n \]
and their motives determine inductively the motives in the Donaldson-Thomas series. In Proposition~\ref{induction} we will show that
\[
(\mathbb{L}^n-1) \frac{[\mathbb{M}^W_{m,n}(0)]-[\mathbb{M}^W_{m,n}(1)]}{[GL_n]} \]
is equal to
\[
[\wis{BS}^W_{m,n}(0)]-[\wis{BS}_{m,n}^W(1)] + \sum_{k=1}^{n-1} \frac{\mathbb{L}^{(m-1)k(n-k)}}{[GL_{n-k}]} ([\wis{BS}^W_{m,k}(0)]-[\wis{BS}^W_{m,k}(1)])([\mathbb{M}^W_{m,k}(0)]-[\mathbb{M}^W_{m,k}(1)]) \]
In section~4 we will compute the first two terms of $U_W(t)$ in the case of the quantized $3$-space in a variety of ways. In the final section we repeat the computation for the homogenized Weyl algebra  and show that it coincides with the conjectured expression of \cite{Cazz}. In a forthcoming paper \cite{LBsuper} we will compute the first two terms of the series for elliptic Sklyanin algebras both in the generic case and the case of $2$-torsion points.

\vskip 4mm

{\em Acknowledgement : } I would like to thank Brent Pym for stimulating conversations concerning the results of \cite{Cazz} and Balazs Szendr\"oi for explaining the importance of the monodromy action (which was lacking in a previous version)  and for sharing his calculations on the Exp-expressions of \cite{Cazz}. I am grateful to Ben Davison for pointing out a computational error in summing up the terms in the homogenized Weyl algebra case and explaining the equality with the conjectured motive.

\section{Brauer-Severi motives}

With $\mathbb{T}_{m,n}$ we will denote the {\em trace ring of $m$ generic $n \times n$ matrices}. That is, $\mathbb{T}_{m,n}$ is the $\C$-subalgebra of the full matrix-algebra $M_n(\C[x_{ij}(k)~|~1 \leq i,j \leq n, 1 \leq k \leq m])$ generated by the $m$ generic matrices
\[
X_k = \begin{bmatrix} x_{11}(k) & \hdots & x_{1n}(k) \\
\vdots & & \vdots \\
x_{n1}(k) & \hdots & x_{nn}(k) \end{bmatrix} \]
together with all elements of the form $Tr(M) 1_n$ where $M$ runs over all monomials in the $X_i$. These algebras have been studied extensively by ringtheorists in the 80ties and some of the results are summarized in the following result

\begin{proposition} Let $\mathbb{T}_{m,n}$ be the trace ring of $m$ generic $n \times n$ matrices, then
\begin{enumerate}
\item{$\mathbb{T}_{m,n}$ is an affine Noetherian domain with center $Z(\mathbb{T}_{m,n})$ of dimension $(m-1) n^2+1$ and generated as $\C$-algebra by the $Tr(M)$ where $M$ runs over all monomials in the generic matrices $X_k$.}
\item{$\mathbb{T}_{m,n}$ is a maximal order and a noncommutative UFD, that is all twosided prime ideals of height one are generated by a central element and $Z(\mathbb{T}_{m,n})$ is a commutative UFD which is a complete intersection if and only if $n=1$ or $(m,n)=(2,2),(2,3)$ or $(3,2)$.}
\item{$\mathbb{T}_{m,n}$ is a reflexive Azumaya algebra unless $(m,n)=(2,2)$, that is, every localization at a central height one prime ideal is an Azumaya algebra.}
\end{enumerate}
\end{proposition}

\begin{proof} For (1) see for example \cite{Procesi} or \cite{Razmyslov}. For (2) see for example \cite{LBAS}, for (3) for example \cite{LBQuiver}. 
\end{proof}

A Cayley-Hamilton algebra of degree $n$ is a $\C$-algebra $A$ , equipped with a linear trace map $tr : A \rTo A$ satisfying the following properties:
\begin{enumerate}
\item{$tr(a).b = b. tr(a)$}
\item{$tr(a.b) = tr(b.a)$}
\item{$tr(tr(a).b) = tr(a).tr(b)$}
\item{$tr(a) = n$}
\item{$\chi_a^{(n)}(a)=0$ where $\chi_a^{(n)}(t)$ is the formal Cayley-Hamilton polynomial of degree $n$, see \cite{ProcesiCH}}
\end{enumerate}

For a Cayley-Hamilton algebra $A$ of degree $n$ it is natural to look at the scheme $\wis{trep}_n(A)$ of all {\em trace preserving} $n$-dimensional representations of $A$, that is, all trace preserving algebra maps $A \rTo M_n(\C)$. A Cayley-Hamilton algebra $A$ of degree $n$ is said to be a {\em smooth Cayley-Hamilton algebra} if $\wis{trep}_n(A)$ is a smooth variety. Procesi has shown that these are precisely the algebras having the smoothness property of allowing lifts modulo nilpotent ideals in the category of all Cayley-Hamilton algebras of degree $n$, see \cite{ProcesiCH}. The \'etale local structure of smooth Cayley-Hamilton algebras and their centers have been extensively studied in \cite{LBbook}.

\begin{proposition} Let $W$ be a homogeneous superpotential in $m$ variables and define the algebra
\[
\mathbb{T}_{m,n}^W(\lambda) = \frac{\mathbb{T}_{m,n}}{(Tr(W)-\lambda)} \quad \text{then} \quad \mathbb{M}_{m,n}^W(\lambda) = \wis{trep}_n(\mathbb{T}_{m,n}^W(\lambda)) \]
If $Tr(W)-\lambda$ is irreducible in the UFD $Z(\mathbb{T}_{m,n})$, then for  $\lambda \not= 0$
\begin{enumerate}
\item{$\mathbb{T}_{m,n}^W(\lambda)$ is a reflexive Azumaya algebra.}
\item{$\mathbb{T}_{m,n}^W(\lambda)$ is a smooth Cayley-Hamilton algebra of degree $n$ and of Krull dimension $(m-1)n^2$.}
\item{$\mathbb{T}_{m,n}^W(\lambda)$ is a domain.}
\item{The central singular locus is the the non-Azumaya locus of $\mathbb{T}_{m,n}^W(\lambda)$ unless $(m,n)=(2,2)$.}
\end{enumerate}
\end{proposition}

\begin{proof}
(1) : As $\mathbb{M}_{m,n}^W(\lambda)=\wis{trep}_n(\mathbb{T}_{m,n}^W(\lambda))$ is a smooth affine variety for $\lambda \not= 0$ (due to homogeneity of $W$) on which $GL_n$ acts by automorphisms, we know that the ring of invariants,
\[
\C[\wis{trep}_n(\mathbb{T}_{m,n}^W(\lambda))]^{GL_n} = Z(\mathbb{T}_{m,n}^W(\lambda)) \]
which coincides with the center of $\mathbb{T}_{m,n}^W(\lambda)$ by e.g. \cite[Prop. 2.12]{LBbook}, is a normal domain. Because the non-Azumaya locus of $\mathbb{T}_{m,n}$ has codimension at least $3$ (if $(m,n) \not= (2,2)$) by \cite{LBQuiver}, it follows that all localizations of $\mathbb{T}_{m,n}^W(\lambda)$ at height one prime ideals are Azumaya algebras. Alternatively, using (2) one can use the theory of local quivers as in \cite{LBbook}.

(2) : That the Cayley-Hamilton degree of the quotient $\mathbb{T}_{m,n}^W(\lambda)$ remains $n$ follows from the fact that $\mathbb{T}_{m,n}$ is a reflexive Azumaya algebra and irreducibility of $Tr(W)-\lambda$. Because $\mathbb{M}_{m,n}^W(\lambda)=\wis{trep}_n(\mathbb{T}_{m,n}^W(\lambda))$ is a smooth affine variety, $\mathbb{T}_{m,n}^W(\lambda)$ is a smooth Cayley-Hamilton algebra. The statement on Krull dimension follows from the fact that the Krull dimension of $\mathbb{T}_{m,n}$ is known to be $(m-1)n^2+1$.

(3) : After taking determinants, this follows from factoriality of $Z(\mathbb{T}_{m,n})$ and irreducibility of $Tr(W)-\lambda$.

(4) : This follows from the theory of local quivers as in \cite{LBbook}. The most general non-simple representations are of representation type $(1,a;1,b)$ with the dimensions of the two simple representations $a,b$ adding up to $n$. The corresponding local quiver is
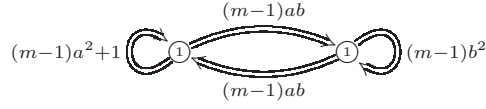
\[
\xymatrix{\vtx{1} \ar@2@/^2ex/[rr]^{(m-1)ab} \ar@{=>}@(ld,lu)^{(m-1)a^2+1} & & \vtx{1} \ar@2@/^2ex/[ll]^{(m-1)ab} \ar@{=>}@(ru,rd)^{(m-1)b^2} }
\]
and as $(m-1)ab \geq 2$ under the assumptions, it follows that the corresponding singular point is singular.
\end{proof}

Let us define for all $k \leq n$ and all $\lambda \in \C$ the locally closed subscheme of $\C^n \times \wis{trep}_n(\mathbb{T}_{m,n}^W(\lambda))$
\[
\wis{X}_{k,n,\lambda} = \{ (v,\phi) \in \C^n \times \wis{trep}_n(\mathbb{T}_{m,n}^W(\lambda))~|~dim_{\C}(\phi(\mathbb{T}_{m,n}^W(\lambda)).v) = k \} \]
Sending a point $(v,\phi)$ to the point in the Grassmannian $\wis{Gr}(k,n)$ determined by the $k$-dimensional subspace $V=\phi(\mathbb{T}_{m,n}^W(\lambda)).v \subset \C^n$ we get a Zariskian fibration as in \cite{Morrison}
\[
\wis{X}_{k,n,\lambda} \rOnto \wis{Gr}(k,n) \]
To compute the fiber over $V$ we choose a basis of $\C^n$ such that the first $k$ base vectors span $V=\phi(\mathbb{T}_{m,n}^W(\lambda)).v$. With respect to this basis, the images of the generic matrices $X_i$ all are of the following block-form
\[
\phi(X_i) = \begin{bmatrix} \phi_k(X_i) & \sigma(X_i) \\ 0 & \phi_{n-k}(X_i) \end{bmatrix} \quad \text{with} \quad \begin{cases} \phi_k(X_i) \in M_k(\C) \\ \phi_{n-k}(X_i) \in M_{n-k}(\C) \\ \sigma(X_i) \in M_{n-k \times k}(\C) \end{cases} \]
Using these matrix-form it is easy to see that
\[
Tr(\phi(W(X_1,\hdots,X_m)))= Tr(\phi_k(W(X_1,\hdots,X_m))) + Tr(\phi_{n-k}(W(X_1,\hdots,X_m))) \]
That is, if $\phi_k \in \wis{trep}_k(\mathbb{T}_{m,k}^W(\mu))$ then $\phi_{n-k} \in \wis{trep}(\mathbb{T}_{m,n-k}^W(\lambda-\mu))$ and moreover we have that $(v,\phi_k) \in \wis{X}_{k,k,\mu}$. Further, the $m$ matrices $\sigma(X_i) \in M_{n-k \times k}(\C)$ can be taken  arbitrary. Rephrasing this in motives we get
\[
[ \wis{X}_{k,n,\lambda} ] = \mathbb{L}^{mk(n-k)} [ \wis{Gr}(k,n) ] \sum_{\mu \in \C} [ \wis{X}_{k,k,\mu} ] [ \wis{trep}_{n-k}(\mathbb{T}_{m,n-k}(\lambda - \mu)) ] \]
Here the summation $\sum_{\mu \in \C}$ is shorthand for distinguishing between zero and non-zero values of $\mu$ and $\lambda-\mu$.  For example, with $\sum_{\mu \in \C} [ \wis{X}_{k,k,\mu} ] [ \wis{trep}_{n-k}(\mathbb{T}_{m,n-k}(\lambda - \mu)) ]$ we mean for $\lambda \not= 0$
\[
 (\mathbb{L}-2)[ \wis{X}_{k,k,1} ] [ \wis{trep}_{n-k}(\mathbb{T}_{m,n-k}(1)) ]+[ \wis{X}_{k,k,0} ] [ \wis{trep}_{n-k}(\mathbb{T}_{m,n-k}(\lambda)) ]+[ \wis{X}_{k,k,\lambda} ] [ \wis{trep}_{n-k}(\mathbb{T}_{m,n-k}(0)) ] \]
 and when $\lambda=0$
 \[
 (\mathbb{L}-1)[ \wis{X}_{k,k,1} ] [ \wis{trep}_{n-k}(\mathbb{T}_{m,n-k}(1)) ]+[ \wis{X}_{k,k,0} ] [ \wis{trep}_{n-k}(\mathbb{T}_{m,n-k}(0)) ]. \]
Further, we have
\[
[ \wis{Gr}(k,n) ] = \frac{[ GL_n ]}{[GL_k ] [GL_{n-k} ] \mathbb{L}^{k(n-k)}} \quad \text{and} \quad [\wis{X}_{k,k,\mu}] = [GL_k] [\wis{BS}_{m,k}^W(\mu)] \]
and substituting this in the above, and recalling that $\mathbb{M}_{m,l}^W(\alpha) = \wis{trep}_l(\mathbb{T}_{m,l}^W(\alpha))$, we get

\begin{proposition} \label{formula} With notations as before we have for all $0 < k < n$ and all $\lambda \in \C$ that
\[
[ \wis{X}_{k,n,\lambda} ] = [GL_n] \mathbb{L}^{(m-1)k(n-k)} \sum_{\mu \in \C} [ \wis{BS}_{m,k}^W(\mu) ] \frac{[\mathbb{M}_{m,n-k}^W(\lambda-\mu) ]}{[GL_{n-k} ]} \]
Further, we have
\[
[ \wis{X}_{0,n,\lambda} ] = [ \mathbb{M}_{m,n}^W(\lambda) ] \quad \text{and} \quad [ \wis{X}_{n,n,\lambda} ] = [GL_n] [ \wis{BS}_{m,n}^W(\lambda) ] \]
\end{proposition}

We can also express this in terms of generating series. Equip the commutative ring $\mathcal{M}_{\C}[[t]]$ with the modified product
\[
t^a \ast t^b = \mathbb{L}^{(m-1)ab} t^{a+b} \]
and consider the following two generating series for all $\frac{1}{2} \not= \lambda \in \C$
\[
\wis{B}_{\lambda}(t) = \sum_{n=1}^{\infty} [ \wis{BS}_{m,n}^W(\lambda) ] t^n \quad \text{and} \quad \wis{R}_{\lambda}(t) = \sum_{n=1}^{\infty} \frac{[ \mathbb{M}_{m,n}^W(\lambda) ]}{[ GL_n ]} t^n \]
\[
\wis{B}_{\frac{1}{2}}(t) = \sum_{n=0}^{\infty} [ \wis{BS}_{m,n}^W(\frac{1}{2}) ] t^n \quad \text{and} \quad \wis{R}_{\frac{1}{2}}(t) = \sum_{n=0}^{\infty} \frac{[ \mathbb{M}_{m,n}^W(\frac{1}{2}) ]}{[ GL_n ]} t^n \]

\begin{proposition} With notations as before we have the functional equation
\[
1+ \wis{R}_{1}(\mathbb{L} t) = \sum_{\mu} \wis{B}_{\mu}(t) \ast \wis{R}_{1-\mu}(t) \]
\end{proposition}

\begin{proof} The disjoint union of the strata of the dimension function on $\C^n \times \wis{trep}_n(\mathbb{T}_{m,n}^W(\lambda))$ gives
\[
\C^n \times \mathbb{M}_{m,n}^W(\lambda) = \wis{X}_{0,n,\lambda} \sqcup \wis{X}_{1,n,\lambda} \sqcup \hdots \sqcup \wis{X}_{n,n,\lambda} \]
Rephrasing this in terms of motives gives
\[
\mathbb{L}^n [ \mathbb{M}_{m,n}^W(\lambda) ] = [ \mathbb{M}_{m,n}^W(\lambda)] +  \sum_{k=1}^{n-1} [ \wis{X}_{k,n,\lambda} ] + [GL_n][\wis{BS}_{m,n}^W(\lambda)]  \]
and substituting the formula of proposition~\ref{formula} into this we get
\[
\frac{[\mathbb{M}_{m,n}^W(\lambda)]}{[GL_n]} \mathbb{L}^n t^n = \frac{[\mathbb{M}_{m,n}^W(\lambda)]}{[GL_n]} t^n + \]
\[
\sum_{k=1}^{n-1} \sum_{\mu \in \C} ([\wis{BS}_{m,k}^W(\mu)] t^k) \ast ( \frac{[ \mathbb{M}_{m,n-k}^W(\lambda-\mu) ]}{[ GL_{n-k} ]} t^{n-k}) + [\wis{BS}_{m,n}^W(\lambda) ] t^n \]
Now, take $\lambda = 1$ then on the left hand side we have the $n$-th term of the series $1+ \wis{R}_{1}(\mathbb{L} t)$ and on the right hand side we have the $n$-th factor of the series $\sum_{\mu} \wis{B}_{\mu}(t) \ast \wis{R}_{1 - \mu}(t)$. The outer two terms arise from the product $\wis{B}_{\frac{1}{2}}(t) \ast \wis{R}_{\frac{1}{2}}(t)$, using that $W$ is homogeneous whence for all $\lambda \not= 0$
\[
\wis{BS}_{m,n}^W(\lambda) \simeq \wis{BS}_{m,n}^W(1) \quad \text{and} \quad \mathbb{M}_{m,n}^W(\lambda) \simeq \mathbb{M}_{m,n}^W(1) \] 
This finishes the proof.
\end{proof}

These formulas allow us to determine the motive $[ \mathbb{M}_{m,n}^W(\lambda) ]$ inductively from lower dimensional contributions and from the knowledge of the motive of the Brauer-Severi scheme $[ \wis{BS}_{m,n}^W(\lambda) ]$.

\begin{proposition} \label{induction} For all $n$ we have the following inductive description of the motives in the Donalson-Thomas series
\[
(\mathbb{L}^n-1) \frac{[\mathbb{M}^W_{m,n}(0)]-[\mathbb{M}^W_{m,n}(1)]}{[GL_n]} \]
is equal to
\[
[\wis{BS}^W_{m,n}(0)]-[\wis{BS}_{m,n}^W(1)] + \sum_{k=1}^{n-1} \frac{\mathbb{L}^{(m-1)k(n-k)}}{[GL_{n-k}]} ([\wis{BS}^W_{m,k}(0)]-[\wis{BS}^W_{m,k}(1)])([\mathbb{M}^W_{m,k}(0)]-[\mathbb{M}^W_{m,k}(1)]) \]
\end{proposition}

\begin{proof} Follows from Proposition~\ref{formula} and the fact that for all $\mu \not= 0$ we have that $[\mathbb{M}_{m,k}^W(\mu)]=[\mathbb{M}_{m,k}^W(1)]$ and $[\wis{BS}_{m,k}^W(\mu)]=[\wis{BS}_{m,k}^W(1)]$.
\end{proof}

\section{Deformations of affine $3$-space}

The commutative polynomial ring $\C[x,y,z]$ is the Jacobi algebra associated with the superpotential $W=XYZ-XZY$. For this reason we restrict in the rest of this paper to cases where the superpotential $W$ is a cubic necklace in three non-commuting variables $X,Y$ and $Z$, that is $m=3$ from now on. As even in this case the calculations become quickly unmanageable we restrict to $n \leq 2$, that is we only will compute the coefficients of $t$ and $t^2$ in $U_W(t)$. We will have to compute the motives of fibers of the Chern-Simons functional
\[
M_2(\C) \oplus M_2(\C) \oplus M_2(\C) \rTo^{Tr(W)} \C \]
so we want to express $Tr(W)$ as a function in the variables of the three generic $2 \times 2$ matrices
\[
X = \begin{bmatrix} n & p \\ q & r \end{bmatrix},~Y=\begin{bmatrix} s & t \\ u & v \end{bmatrix},~Z= \begin{bmatrix} w & x \\ y & z \end{bmatrix}. \]
We will call $\{ n,r,s,v,w,x \}$ (resp. $\{ p,t,x \}$ and $\{ q,u,y \}$) the diagonal- (resp. upper- and lower-) variables. We claim that
\[
Tr(W) = C + Q_q.q + Q_u.u + Q_y.y \]
where $C$ is a cubic in the diagonal variables and $Q_q,Q_u$ and $Q_y$ are bilinear in the diagonal and upper variables, that is, there are linear terms $L_{ab}$ in the diagonal variables such that
\[
\begin{cases}
Q_q = L_{qp}.p+L_{qt}.t+L_{qx}.x \\
Q_u = L_{up}.p+L_{ut}.t+L_{ux}.x \\
Q_y = L_{yp}.p+L_{yt}.t+L_{yx}.x
\end{cases}
\]
This follows from considering the two diagonal entries of a $2 \times 2$ matrix as the vertices of a quiver and the variables as arrows connecting these vertices as follows
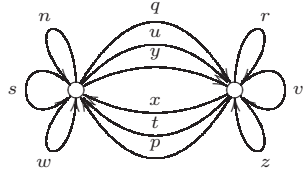
\[
\xymatrix{\vtx{} \ar@(u,ul)_n \ar@(ul,dl)_s \ar@(dl,d)_w \ar@/^6ex/[rr]^q \ar@/^4ex/[rr]^u \ar@/^2ex/[rr]^y & & \vtx{} \ar@/^6ex/[ll]_p \ar@/^4ex/[ll]_t \ar@/^2ex/[ll]_x \ar@(u,ur)^r \ar@(ur,dr)^v \ar@(dr,d)^z} \]
and observing that only an oriented path of length $3$ starting and ending in the same vertex can contribute something non-zero to $Tr(W)$. Clearly these linear and cubic terms are fully determined by $W$. If we take
\[
W = \alpha X^3 + \beta Y^3 + \gamma Z^3 + \delta XYZ + \epsilon XZY \]
then we have $C = W(n,s,w)+W(r,v,z)$ and
\[
\begin{cases}
L_{qp} &= 3 \alpha(n+r) \\
L_{qt} &= \epsilon w + \delta z \\
L_{qx} &= \delta s + \epsilon v
\end{cases} \quad
\begin{cases}
L_{up} &= \delta w + \epsilon z \\
L_{ut} &= 3 \beta(s+v) \\
L_{ux} &= \epsilon n + \delta r \\
\end{cases} \quad
\begin{cases}
L_{yp} &= \epsilon s + \delta v \\
L_{yt} &= \delta n + \epsilon r \\
L_{yx} &= 3 \gamma(w+z) \\
\end{cases}
\]
By using the cellular decomposition of the Brauer-Severi scheme of $\mathbb{T}_{3,2}$ one can simplify the computations further by specializing certain variables. From \cite{ReinekeBS} we deduce that $\wis{BS}_2(\mathbb{T}_{3,2})$ has a cellular decomposition as $\mathbb{A}^{10} \sqcup \mathbb{A}^8 \sqcup \mathbb{A}^8$ where the three cells have representatives
\[
\begin{cases}
\wis{cell}_1~:~v = \begin{bmatrix} 1 \\ 0 \end{bmatrix}, \quad X = \begin{bmatrix} 0 & p \\ 1 & r \end{bmatrix}, \quad 
Y = \begin{bmatrix} s & t \\ u & v \end{bmatrix}, \quad
Z = \begin{bmatrix} w & x \\ y & z \end{bmatrix}  \\
\\
\wis{cell}_2~:~v = \begin{bmatrix} 1 \\ 0 \end{bmatrix}, \quad X = \begin{bmatrix} n & p \\ 0 & r \end{bmatrix}, \quad 
Y = \begin{bmatrix} 0 & t \\ 1 & v \end{bmatrix}, \quad
Z = \begin{bmatrix} w & x \\ y & z \end{bmatrix} \\
\\
\wis{cell}_3~:~v = \begin{bmatrix} 1 \\ 0 \end{bmatrix}, \quad X = \begin{bmatrix} n & p \\ 0 & r \end{bmatrix}, \quad 
Y = \begin{bmatrix} s & t \\ 0 & v \end{bmatrix}, \quad
Z = \begin{bmatrix} 0 & x \\ 1 & z \end{bmatrix} 
\end{cases}
\]
It follows that $\wis{BS}_{3,2}^W(1)$ decomposes as $\mathbf{S_1} \sqcup \mathbf{S_2} \sqcup \mathbf{S_3}$ where the subschemes $\mathbf{S_i}$ of $\mathbb{A}^{11-i}$ have defining equations
\[
\begin{cases}
\mathbf{S_1}~:~(C + Q_u.u + Q_y.y + Q_q)|_{n=0} = 1 \\
\mathbf{S_2}~:~(C + Q_y.y + Q_u)|_{s=0} = 1 \\
\mathbf{S_3}~:~(C + Q_y)|_{w=0} = 1
\end{cases}
\]
Note that in using the cellular decomposition, we set a variable equal to $1$. So, in order to retain a homogeneous form we let $\mathbb{G}_m$ act on $n,s,w,r,v,z$ with weight one, on $q,u,y$ with weight two and on $x,t,p$ with weight zero. Thus, we need a slight extension of \cite[Thm. 1.3]{Davison} as to allow $\mathbb{G}_m$ to act with weight two on certain variables.

From now on we will assume that $W$ is as above with $\delta=1$ and $\epsilon \not= 0$. In this generality we can prove:

\begin{proposition} \label{S3} With assumptions as above
\[
[ \mathbf{S_3} ] = \begin{cases}
\mathbb{L}^7-\mathbb{L}^4+\mathbb{L}^3 [ W(n,s,0)+W(-\epsilon^{-1} n,- \epsilon s,0) = 1]_{\mathbb{A}^2}  & \text{if $\gamma \not= 0$} \\
\mathbb{L}^7 - \mathbb{L}^5 + \mathbb{L}^3 [W(n,s,0) + W(-\epsilon^{-1} n,-\epsilon s,z) = 1]_{\mathbb{A}^3} & \text{if $\gamma=0$} \end{cases} 
\]
\end{proposition}

\begin{proof}
$\mathbf{S_3}$ : The defining equation in $\mathbb{A}^8$ is equal to
\[
W(n,s,0)+W(r,v,z)+(\epsilon s +v)p+(n+\epsilon r)t+3 \gamma(z)x = 1 \]
If $\epsilon s + v \not= 0$ we can eliminate $p$ and get a contribution $\mathbb{L}^5(\mathbb{L}^2-\mathbb{L})$. If $v = - \epsilon s$ but $n + \epsilon r \not= 0$ we can eliminate $t$ and get a term $\mathbb{L}^4(\mathbb{L}^2-\mathbb{L})$. From now on we may assume that $v = -\epsilon s$ and $r= - \epsilon^{-1}n$. 

\noindent
$\gamma \not = 0$ : Assume first that $z \not= 0$ then we can eliminate $x$ and get a contribution $\mathbb{L}^4(\mathbb{L}-1)$. If $z=0$ then we get a term
\[
\mathbb{L}^3 [ W(n,s,0)+W(-\epsilon^{-1} n,- \epsilon s,0) = 1]_{\mathbb{A}^2}  \]

\noindent
$\gamma = 0$ : Then we have a remaining contribution
\[
\mathbb{L}^3 [W(n,s,0) + W(-\epsilon^{-1} n,-\epsilon s,z) = 1]_{\mathbb{A}^3}  \]
Summing up all contributions gives the result.
\end{proof}

Calculating the motives of $\mathbf{S_2}$ and $\mathbf{S_1}$ in this generality quickly leads to a myriad of subcases to consider. For this reason we will defer the calculations in the cases of interest to the next sections. Specializing Proposition~\ref{induction} to the case of $n=2$ we get

\begin{proposition} \label{case2} For $n=2$ we have that
\[
(\mathbb{L}^2-1) \frac{[\mathbb{M}^W_{3,2}(0)]-[\mathbb{M}^W_{3,2}(1)]}{[GL_2]} \]
is equal to
\[
[\wis{BS}^W_{3,2}(0)]-[\wis{BS}^W_{3,2}(1)]+\frac{\mathbb{L}^2}{(\mathbb{L}-1)}([\mathbb{M}^W_{3,1}(0)]-[\mathbb{M}^W_{3,1}(1)])^2 \]
\end{proposition}

\begin{proof} 
The result follows from Proposition~\ref{induction} and from the fact that $\mathbf{BS}_{3,1}^W(1)=\mathbb{M}_{3,1}^W(1)$ and $\mathbf{BS}_{3,1}^W(0)=\mathbb{M}_{3,1}^W(0)$.
\end{proof}

\section{Quantum affine three-space}

For $q \in \C^*$ consider the superpotential $W_q = XYZ-qXZY$, then the associated algebra $R_{W_q}$ is the quantum affine $3$-space
\[
R_{W_q} = \frac{\C \langle X,Y,Z \rangle}{(XY-qYX,ZX-qXZ,YZ-qZY)} \]
It is well-known that $R_{W_q}$ has finite dimensional simple representations of dimension $n$ if and only if $q$ is a primitive $n$-th root of unity. For other values of $q$ the only finite dimensional simples are $1$-dimensional and parametrized by $XYZ=0$ in $\mathbb{A}^3$. In this case we have
\[
\begin{cases}
[\mathbb{M}_{3,1}^{W_q}(1)]=[(q-1)XYZ=1]_{\mathbb{A}^3} = (\mathbb{L}-1)^2 \\
 [\mathbb{M}_{3,1}^{W_q}(0)]=[(1-q)XYZ=0]_{\mathbb{A}^3} = 3 \mathbb{L}^2 - 3 \mathbb{L} + 1
 \end{cases}
  \]
That is, the coefficient of $t$ in $U_{W_q}(t)$ is equal to
\[
\mathbb{L}^{-1} \frac{[\mathbb{M}_{3,1}^{W_q}(0)- [\mathbb{M}_{3,1}^{W_q}(1)]}{[GL_1]} = \mathbb{L}^{-1} \frac{2 \mathbb{L}^2-\mathbb{L}}{\mathbb{L}-1} = \frac{2 \mathbb{L}-1}{\mathbb{L}-1} \]
In \cite[Thm. 3.1]{Cazz} it is shown that in case $q$ is not a root of unity, then
\[
U_{W_q}(t) = \wis{Exp}(\frac{2 \mathbb{L}-1}{\mathbb{L}-1} \frac{t}{1-t}) \]
and if $q$ is a primitive $n$-th root of unity then
\[
U_{W_q}(t) = \wis{Exp}(\frac{2\mathbb{L}-1}{\mathbb{L}-1} \frac{t}{1-t} + (\mathbb{L}-1) \frac{t^n}{1-t^n}) \]
In \cite[3.4.1]{Cazz} a rather complicated attempt is made to explain the term $\mathbb{L}-1$ in case $q$ is an $n$-th root of unity in terms of certain simple $n$-dimensional representations of $R_{W_q}$. Note that the geometry of finite dimensional representations of the algebra $R_{W_q}$ is studied extensively in \cite{Kevin2} and note that there are additional simple $n$-dimensional representations not taken into account in \cite[3.4.1]{Cazz}.

Perhaps a more conceptual explanation of the two terms in the exponential expression of $U_{W_q}(t)$ in case $q$ is an $n$-th root of unity is as follows. As $W_q$ admits a cut $W_q=X(YZ-qZY)$ it follows from \cite{Morrison} that for all dimensions $m$ we have
\[
[\mathbb{M}_{3,m}^{W_q}(0)]-[\mathbb{M}_{3,m}^{W_q}(1)] = \mathbb{L}^{m^2} [ \wis{rep}_m(\C_q[Y,Z])] \]
where $\C_q[Y,Z]=\C \langle Y,Z \rangle/(YZ-qZY)$ is the quantum plane. If $q$ is an $n$-th root of unity the only finite dimensional simple representations of $\C_q[Y,Z]$ are of dimension $1$ or $n$. The $1$-dimensional simples are parametrized by $YZ=0$ in $\mathbb{A}^2$ having as motive $2 \mathbb{L}-1$ and as all have $GL_1$ as stabilizer group, this explains the term $(2 \mathbb{L}-1)/(\mathbb{L}-1)$. The center of $\C_q[Y,Z]$ is equal to $\C[Y^n,Z^n]$ and the corresponding variety $\mathbb{A}^2=\wis{Max}(\C[Y^n,Z^n])$ parametrizes $n$-dimensional semi-simple representations.The $n$-dimensional simples correspond to the Zariski open set $\mathbb{A}^2 - (Y^nZ^n=0)$ which has as motive $(\mathbb{L}-1)^2$. Again, as all these have as $GL_2$-stabilizer subgroup $GL_1$, this explains the term
\[
\mathbb{L}-1 = \frac{(\mathbb{L}-1)^2}{[GL_1]} \]
As the superpotential allows a cut in this case we can use the full strength of \cite{Behrend}and can obtain $[\mathbb{M}^W_{3,2}(0)]$ from $[\mathbb{M}^W_{3,2}(1)]$ from the equality \[
\mathbb{L}^{12} = [\mathbb{M}^W_{3,2}(0)] + (\mathbb{L}-1)[\mathbb{M}^W_{3,2}(1)] \]

To illustrate the inductive procedure using Brauer-Severi motives we will consider the case $n=2$, that is $q=-1$ with superpotential $W=XYZ+XZY$. In this case we have from \cite[Thm. 3.1]{Cazz} that
\[
U_W(t) = \wis{Exp}(\frac{2 \mathbb{L}-1}{\mathbb{L}-1} \frac{t}{1-t}+(\mathbb{L}-1) \frac{t^2}{1-t^2} \]
The basic rules of the plethystic exponential on $\mathcal{M}_{\C}[[t]]$ are
\[
\wis{Exp}(\sum_{n \geq 1} [A_n]t^n) = \prod_{n \geq 1} (1-t^n)^{-[A_n]} \quad \text{where} \quad (1-t)^{-\mathbb{L}^m} = (1-\mathbb{L}^m t)^{-1} \]
and one has to extend all infinite products in $t$ and $\mathbb{L}^{-1}$. One starts by rewriting $U_W(t)$ as a product
\[
U_W(t) = \wis{Exp}(\frac{t}{1-t}) \wis{Exp}(\frac{\mathbb{L}}{\mathbb{L}-1}\frac{t}{1-t}) \wis{Exp}(\frac{\mathbb{L} t^2}{1-t^2}) \wis{Exp}(\frac{t^2}{1-t^2})^{-1} \]
where each of the four terms is an infinite product
\[
 \wis{Exp}(\frac{t}{1-t}) = \prod_{m \geq 1}(1-t^m)^{-1}, \qquad \wis{Exp}(\frac{\mathbb{L}}{\mathbb{L}-1} \frac{t}{1-t}) = \prod_{m \geq 1} \prod_{j \geq 0} (1 - \mathbb{L}^{-j}t^m)^{-1} \]
 \[
  \wis{Exp}(\frac{\mathbb{L} t^2}{1-t^2}) = \prod_{m \geq 1} (1 - \mathbb{L}t^{2m})^{-1}, \qquad \wis{Exp}(\frac{t^2}{1-t^2})^{-1} = \prod_{m \geq 1} (1-t^{2m} \]
  That is, we have to work out the infinite product
  \[
  \prod_{m \geq 1} ((1-t^{2m-1})^{-1} (1 - \mathbb{L} t^{2m})^{-1}) \prod_{m \geq 1} \prod_{j \geq 0} (1- \mathbb{L}^{-j} t^m)^{-1} \]
  as a power series in $t$, at least up to quadratic terms. One obtains
  \[
U_W(t) = 1 + \frac{2 \mathbb{L}-1}{\mathbb{L}-1} t + \frac{\mathbb{L}^4+3\mathbb{L}^3-2 \mathbb{L}^2 - 2 \mathbb{L}+1}{(\mathbb{L}^2-1)(\mathbb{L}-1)} t^2 + \hdots \]
That is, if $W=XYZ+XZY$ one must have the relation:
\[
[\mathbb{M}_{3,2}^W(0)]-[\mathbb{M}_{3,2}^W(1)] = \mathbb{L}^5(\mathbb{L}^4+3 \mathbb{L}^3-2 \mathbb{L}^2- 2\mathbb{L}+1) \]

\subsection{Dimensional reduction}

It follows from the dimensional reduction argument of \cite{Morrison} that
\[
[\mathbb{M}_{3,2}^W(0)] - [ \mathbb{M}_{3,2}^W(1) ] = \mathbb{L}^4 [ \wis{rep}_2~\C_{-1}[X,Y] ] \]
where $\C_{-1}[X,Y]$ is the quantum plane at $q=-1$, that is, $\C \langle X,Y \rangle / (XY+YX)$.
The matrix equation
\[
\begin{bmatrix} a & b \\ c & d \end{bmatrix} \begin{bmatrix} e & f \\ g & h \end{bmatrix} +  \begin{bmatrix} e & f \\ g & h \end{bmatrix} \begin{bmatrix} a & b \\ c & d \end{bmatrix} = \begin{bmatrix} 0 & 0 \\ 0 & 0 \end{bmatrix} \]
gives us the following system of equations
\[
\begin{cases}
2 ae + bg + fc = 0 \\
2 hd + bg + fc = 0 \\
f (a+d) + b (e+h) = 0 \\
c (h+e) + g (a+d) = 0 
\end{cases}
\]
where the two first are equivalent to $ae=hd$ and $2ae+bg+fc=0$. Changing variables
\[
x=\frac{1}{2}(a+d), \quad y = \frac{1}{2}(a-d), \quad u = \frac{1}{2}(e+h), \quad v= \frac{1}{2}(e-h) \]
the equivalent system then becomes (in the variables $b,c,f,g,u,v,x,y$)
\[
\begin{cases}
xv+yu = 0 \\
xu+yv + bg+fc = 0 \\
fx+bu = 0 \\
cu+gx = 0
\end{cases}
\]

\begin{proposition} The motive of $R_2= \wis{rep}_2~\C_{-1}[x,y]$ is equal to
\[
[ R_2 ] = \mathbb{L}^5 + 3 \mathbb{L}^4 -  2 \mathbb{L}^3 - 2 \mathbb{L}^2 + \mathbb{L} \]
\end{proposition}

\begin{proof}
If $x \not= 0$ we obtain
\[
v=-\frac{yu}{x}, \quad f=-\frac{bu}{x}, \quad g=-\frac{cu}{x} \]
and substituting these in the remaining second equation we get the equation(s)
\[
u(y^2-x^2+2bc)=0 \quad \text{and} \quad x \not= 0 \]
If  $u \not= 0$  then $y^2-x^2+2bc=0$. If in addition $b \not= 0$ then $c = \tfrac{x^2-y^2}{2b}$ and $y$ is free. As $x,u$ and $b$ are non-zero this gives a contribution $(\mathbb{L}-1)^3 \mathbb{L}$. 
If $b=0$ then $c$ is free and $x^2-y^2=0$, so $y = \pm x$. This together with $x \not= 0 \not= u$ leads to a contribution of
$2 \mathbb{L}(\mathbb{L}-1)^2$. If $u = 0$ then $y,b$ and $c$ are free variables, and together with $x \not= 0$ this gives
$(\mathbb{L}-1) \mathbb{L}^3$.

\vskip 3mm
\noindent
Remains the case that $x = 0$. Then the system reduces to
\[
\begin{cases}
yu = 0 \\
yv+bg+fc = 0 \\
bu = 0 \\
cu = 0
\end{cases}
\]
If $u \not= 0$ then $y=0, b=0$ and $c=0$ leaving $c,g,v$ free. This gives
$(\mathbb{L}-1) \mathbb{L}^3$. 
If $u = 0$  then the only remaining equation is $yv+bg+fc=0$. That is, we get the cone in $\mathbb{A}^6$ of the Grassmannian $Gr(2,4)$ in $\mathbb{P}^5$. As the motive of $Gr(2,4)$ is
\[
[Gr(2,4)] = (\mathbb{L}^2+1)(\mathbb{L}^2+\mathbb{L}+1) \]
we get a contribution of
\[
(\mathbb{L}-1)(\mathbb{L}^2+1)(\mathbb{L}^2+\mathbb{L}+1) + 1 \]
Summing up all contributions gives the  desired result.
\end{proof}

\subsection{Brauer-Severi motives} In the three cells of the Brauer-Severi scheme of $\mathbb{T}_{3,2}$ of dimensions resp. $10,9$ and $8$ the superpotential $Tr(XYZ+XZY)$ induces the equations:
\[
\begin{cases}
\mathbf{S_1}~:~2rvz+puz+pvy+rty+psy+rux+puw+tz+vx+sx+tw=1 \\
\mathbf{S_2}~:~2rvz+pvy+rty+nty+pz+rx+nx+pw=1 \\
\mathbf{S_3}~:~2rvz+pv+rt+nt+ps=1
\end{cases}
\]

\begin{proposition} With notations as above, the Brauer-Severi scheme of $\mathbb{T}_{3,2}^W(1)$ has a decomposition
\[
\mathbf{BS}_{3,2}^W(1) = \mathbf{S_1} \sqcup \mathbf{S_2} \sqcup \mathbf{S_3} \]
where the schemes $\mathbf{S_i}$ have motives
\[
\begin{cases}
[ \mathbf{S_1} ] = \mathbb{L}^9-\mathbb{L}^6-2\mathbb{L}^5+3\mathbb{L}^4-\mathbb{L}^3 \\
[ \mathbf{S_2} ] = \mathbb{L}^8- 2 \mathbb{L}^5 + \mathbb{L}^4 \\
[ \mathbf{S_3} ] = \mathbb{L}^7 - 2 \mathbb{L}^4 + \mathbb{L}^3 \\
\end{cases}
\]
Therefore, the Brauer-Severi scheme has motive
\[
[ \mathbf{BS}_{3,2}^W(1) ] = \mathbb{L}^9+\mathbb{L}^8+\mathbb{L}^7-\mathbb{L}^6-4\mathbb{L}^5+2\mathbb{L}^4 \]
\end{proposition}

\begin{proof} $\mathbf{S_1}$ : From Proposition~\ref{S3} we obtain
\[
[ \mathbf{S_3} ] = \mathbb{L}^7 - \mathbb{L}^5 + \mathbb{L}^3[W(n,s,0)+W(-n,-s,z)=1]_{\mathbb{A}^3} \]
and as $W(n,s,0)+W(-n,-s,z)=2nsz$ we get $\mathbb{L}^7 - \mathbb{L}^5 + \mathbb{L}^3(\mathbb{L}-1)^2$.

\vskip 3mm

\noindent
$\mathbf{S_2}$ : The defining equation is
\[
2 rvz + y (pv + (r+n)t) + p(z+w) + x(r+n) = 1 \]
If $r+n \not= 0$ we can eliminate $x$ and have a contribution $\mathbb{L}^6 (\mathbb{L}^2-\mathbb{L})$. If $r+n=0$ we get the equation
\[
2 rvz + p (yv+z+w) = 1 \]
If $yv+z+w \not= 0$ we can eliminate $p$ and get a term $\mathbb{L}^3(\mathbb{L}^4-\mathbb{L}^3)$. If $r+n=0$ and $yv+z+w=0$ we have $2rvz = 1$ so a term $\mathbb{L}^4(\mathbb{L}-1)^2$. Summing up gives us
\[
[ \mathbf{S}_2 ] = \mathbb{L}^4(\mathbb{L}-1)(\mathbb{L}^3+\mathbb{L}^2+\mathbb{L}-1) = \mathbb{L}^8- 2 \mathbb{L}^5 + \mathbb{L}^4 \]

\vskip 3mm

\noindent
$\mathbf{S_1}$ : The defining equation is
\[
2 rvz + p(u(z+w)+y(v+s))+t(z+w+ry)+x(v+s+ru) = 1 \]
If $v+s+ru \not= 0$ we can eliminate $x$ and get $\mathbb{L}^5(\mathbb{L}^4-\mathbb{L}^3)$. If $v+s+ru=0$ and $z+w+ry \not= 0$ we can eliminate $t$ and have a term $\mathbb{L}^4(\mathbb{L}^4-\mathbb{L}^3)$. If $v+s+ru=0$ and $z+w+ry=0$, the equation becomes (in $\mathbb{A}^8$, with $t,x$ free variables)
\[
2r(vz-puy) = 1 \]
giving a term $\mathbb{L}^2(\mathbb{L}^5-[ vz=puy ])$. To compute $[ vz=puy ]_{\mathbb{A}^5}$ assume first that $v \not= 0$, then this gives $\mathbb{L}^3(\mathbb{L}-1)$ and if $v=0$ we get $\mathbb{L}(3 \mathbb{L}^2-3 \mathbb{L} + 1)$. That is, $[vz=puy]_{\mathbb{A}^5}=\mathbb{L}^4+2 \mathbb{L}^3-3 \mathbb{L}^2+\mathbb{L}$. In total this gives us
\[
[ \mathbf{S}_1 ] = \mathbb{L}^3(\mathbb{L}-1)(\mathbb{L}^5+\mathbb{L}^4+\mathbb{L}^3-2 \mathbb{L}+1) = \mathbb{L}^9-\mathbb{L}^6-2\mathbb{L}^5+3\mathbb{L}^4-\mathbb{L}^3 \]
finishing the proof.
\end{proof}

\begin{proposition} From the Brauer-Severi motive we obtain
\[
\begin{cases}
[ \mathbb{M}_{3,2}^W(1) ] &= \mathbb{L}^{11}-\mathbb{L}^8-3\mathbb{L}^7+2\mathbb{L}^6+2\mathbb{L}^5-\mathbb{L}^4 \\
[ \mathbb{M}_{3,2}^W(0) ] &= \mathbb{L}^{11} + \mathbb{L}^9 + 2 \mathbb{L}^8 - 5\mathbb{L}^7 + 3 \mathbb{L}^5 - \mathbb{L}^4
\end{cases}
\]
As a consequence we have,
\[
[\mathbb{M}_{3,2}^W(0)]-[\mathbb{M}_{3,2}^W(1) ]=\mathbb{L}^4(\mathbb{L}^5+3 \mathbb{L}^4-2\mathbb{L}^3-2\mathbb{L}^2+\mathbb{L}) \]
\end{proposition}

\begin{proof} We have already seen that $\mathbb{M}_{3,1}^W(1) = \{ (x,y,z)~|~2xyz=1 \}$ and $\mathbb{M}_{3,1}^W(0) = \{ (x,y,z)~|~xyz=0 \}$ whence
\[
[\mathbb{M}_{3,1}^W(1)] = (\mathbb{L}-1)^2 \quad \text{and} \quad [ \mathbb{M}_{3,1}^W(0)]=3 \mathbb{L}^2-3\mathbb{L}+1 \]
Plugging this and the obtained Brauer-Severi motive into Proposition~\ref{induction} gives $[\mathbb{M}_{3,2}^W(1)]$. From this $[\mathbb{M}_{3,2}^W(0)]$ follows from the equation $\mathbb{L}^{12} = (\mathbb{L}-1)[\mathbb{M}_{3,2}^W(1)] + [ \mathbb{M}_{3,2}^W(0)]$.
\end{proof}

\section{The homogenized Weyl algebra}

If we consider the superpotential $W=XYZ-XZY- \frac{1}{3}X^3$ then the associated algebra $R_W$ is the homogenized Weyl algebra
\[
R_W = \frac{\C \langle X,Y,Z \rangle}{(XZ-ZX,XY-YX,YZ-ZY-X^2)} \]
In this case we have $\mathbb{M}_{3,1}^W(1) = \{ x^3=-3 \}$ and $\mathbb{M}_{3,1}^W(0) = \{ x^3 = 0 \}$, whence
\[
[\mathbb{M}_{3,1}^W(1)] =  \mathbb{L}^2[\mu_3], \quad \text{and} \quad [\mathbb{M}_{3,1}^W(0)] = \mathbb{L}^2 \]
where, as in \cite[3.1.3]{Cazz} we denote by $[\mu_3]$ the equivariant motivic class of $\{ x^3=1 \} \subset \mathbb{A}^1$ carrying the canonical action of $\mu_3$. Therefore, the coefficient of $t$ in $U_W(t)$ is equal to
\[
\mathbb{L}^{-1} \frac{[\mathbb{M}_{3,1}^W(0)] - [\mathbb{M}_{3,1}^W(0)]}{[GL_1]} =  \frac{\mathbb{L}(1-[\mu_3])}{\mathbb{L}-1} \]
As all finite dimensional simple representations of $R_W$ are of dimension one, this leads to the conjectural expression \cite[Conjecture 3.3]{Cazz}
\[
U_W(t)  \overset{?}{=} \wis{Exp}(\frac{\mathbb{L}(1-[\mu_3])}{\mathbb{L}-1} \frac{t}{1-t}) \]
Balazs Szendr\"oi kindly provided the calculation of the first two terms of this series. Denote with $\tilde{\mathbf{M}} = 1 - [ \mu_3]$, then
\[
U_W(t) \overset{?}{=} 1 + \frac{\mathbb{L} \tilde{\mathbf{M}}}{\mathbb{L}-1}t + \frac{\mathbb{L}^2 \tilde{\mathbf{M}}^2+ \mathbb{L}(\mathbb{L}^2-1) \tilde{\mathbf{M}} + \mathbb{L}^2(\mathbb{L}-1) \sigma_2(\tilde{\mathbf{M}})}{(\mathbb{L}^2-1)(\mathbb{L}-1)} t^2 + \hdots \]
As was pointed out by B. Pym and B. Davison it follows from \cite[Defn 4.4 and Prop 4.5 (4)]{Davison} that $\sigma_2(\tilde{\mathbf{M}}) = \mathbb{L}$, so the second term is equal to
\[
\frac{\mathbb{L}^3(\mathbb{L}-1) + \tilde{\mathbf{M}} \mathbb{L}(\mathbb{L}^2-1) + \tilde{\mathbf{M}}^2 \mathbb{L}^2}{(\mathbb{L}^2-1)(\mathbb{L}-1)} \]

We will now compute the this second term using Brauer-Severi motives.

\vskip 3mm
Recall that $\wis{BS}_{3,2}^W(i)$, for $i=0,1$, decomposes as $\mathbf{S_1} \sqcup \mathbf{S_2} \sqcup \mathbf{S_3}$ where the subschemes $\mathbf{S_i}$ of $\mathbb{A}^{11-i}$ have defining equations
\[
\begin{cases}
\mathbf{S_1}~:~-\frac{1}{3}r^3+((w-z)p+rx)u+((v-s)p-rt)y-rp+(z-w)t+(s-v)x  = \delta_{i1} \\
\mathbf{S_2}~:~-\frac{1}{3}n^3-\frac{1}{3}r^3+(vp+(n-r)t)y + (w-z)p+(r-n)x = \delta_{i1} \\
\mathbf{S_3}~:~-\frac{1}{3}n^3-\frac{1}{3}r^3+(v-s)p+(n-r)t = \delta_{i1}
\end{cases}
\]
If we let the generator of $\mu_3$ act with weight one on the variables $n,s,w,r,v,z$, with weight two on $x,t,p$ and with weight zero on $q,u,y$ we see that the schemes $S_j$ for $i=1$ are indeed $\mu_3$-varieties. We will now compute their equivariant motives:

\begin{proposition} With notations as above, the Brauer-Severi scheme of $\mathbb{T}_{3,2}^W(1)$ has a decomposition
\[
\mathbf{BS}_{3,2}^W(1) = \mathbf{S_1} \sqcup \mathbf{S_2} \sqcup \mathbf{S_3} \]
where the schemes $\mathbf{S_i}$ have equivariant motives
\[
\begin{cases}
[ \mathbf{S_1} ] = \mathbb{L}^9 - \mathbb{L}^6 \\
[ \mathbf{S_2} ] = \mathbb{L}^8 + ([\mu_3] -1) \mathbb{L}^6 = \mathbb{L}^8- \tilde{\mathbf{M}} \mathbb{L}^6  \\
[ \mathbf{S_3} ] = \mathbb{L}^7 + ([\mu_3]-1) \mathbb{L}^5 = \mathbb{L}^7 - \tilde{\mathbf{M}} \mathbb{L}^5  \\
\end{cases}
\]
Therefore, the Brauer-Severi scheme $\mathbf{BS}^W_{3,2}(1)$ has equivariant motive
\[
[ \mathbf{BS}_{3,2}^W(1) ] = \mathbb{L}^9 + \mathbb{L}^8 + \mathbb{L}^7 + ([\mu_3]-2) \mathbb{L}^6 + ([\mu_3]-1) \mathbb{L}^5  \]
\end{proposition}

\begin{proof} $\mathbf{S_3}$ : If $v-s \not= 0$ we can eliminate $p$ and obtain a contribution $\mathbb{L}^5(\mathbb{L}^2-\mathbb{L})$. If $v=s$ and $n-r \not= 0$ we can eliminate $t$ and obtain a term $\mathbb{L}^4(\mathbb{L}^2-\mathbb{L})$. Finally, if $v=s$ and $n=r$ we have the identity $-\frac{2}{3}n^3=1$ and a contribution $\mathbb{L}^5 [ \mu_3 ]$.

\vskip 3mm

\noindent
$\mathbf{S_2}$ : If $r-n \not= 0$ we can eliminate $x$ and get a term $\mathbb{L}^6(\mathbb{L}^2-\mathbb{L})$. If $r-n=0$ we get the equation in $\mathbb{A}^8$
\[
-\frac{2}{3}n^3+p(vy+w-z) = 1 \]
If $vy+w-z \not= 0$ we can eliminate $p$ and get a contribution $\mathbb{L}^3(\mathbb{L}^4-\mathbb{L}^3)$. Finally, if $vy+w-z=0$ we get the equation $-\frac{2}{3}n^3=1$ and hence a term $\mathbb{L}^3. \mathbb{L}^3[\mu_3]$.

\vskip 3mm

\noindent
$\mathbf{S_1}$ : If $(w-z)p+rx \not= 0$ then we can eliminate $u$ and get a contribution
\[
\mathbb{L}^4(\mathbb{L}^5-[(w-z)p+rx=0]_{\mathbb{A}^5}) = \mathbb{L}^6(\mathbb{L}-1)(\mathbb{L}^2-1) \]
If $(w-z)p+rx=0$ but $(v-s)p-rt \not= 0$ we can eliminate $y$ and get a term
\[
\mathbb{L}.[(w-z)p+rx=0,(v-s)p-rt \not= 0]_{\mathbb{A}^8} \]
To compute the equivariant motive in $\mathbb{A}^8$ assume first that $r \not= 0$ then we can eliminate $x$ from the equation and obtain 
\[
\mathbb{L}^2[r \not= 0,(v-s)p-rt \not= 0]_{\mathbb{A}^5}=\mathbb{L}^2(\mathbb{L}^4(\mathbb{L}-1) - [r \not= 0,(v-s)p-rt=0]_{\mathbb{A}^5}) = \mathbb{L}^5(\mathbb{L}-1)^2 \]
If $r=0$ we have to compute $[(w-z)p=0,(v-s)p\not= 0]_{\mathbb{A}^7} = \mathbb{L}^2(\mathbb{L}-1)(\mathbb{L}^2-\mathbb{L})\mathbb{L} = \mathbb{L}^4(\mathbb{L}-1)^2$. So, in total this case gives a contribution
\[
\mathbb{L}.[(w-z)p+rx=0,(v-s)p-rt \not= 0]_{\mathbb{A}^8} = \mathbb{L}^5(\mathbb{L}-1)(\mathbb{L}^2-1) \]
If $(w-z)p+rx=0$, $(v-s)p-rt=0$ and $r \not= 0$ we can eliminate $x = \tfrac{z-w}{r}p$ and $t=\tfrac{v-s}{r}p$ and substituting in the defining equation of $\mathbf{S_1}$  we get
\[
-\frac{1}{3}r^3-rp = 1 \]
so we can eliminate $p$ and obtain a contribution $\mathbb{L}^6(\mathbb{L}-1)$.
Finally, if $(w-z)p+rx=0$, $(v-s)p-rt=0$ and $r = 0$ we get the system of equations
\[
\begin{cases} (w-z)p = 0 \\ (v-s)p = 0 \\ (z-w)t+(s-v)x = 1 \end{cases} \]
If $p \not= 0$ we must have $w-z=0$ and $v-s=0$ which is impossible, so we must have $p=0$ and the remaining equation is $(z-w)t+(s-v)x=1$ giving a contribution $\mathbb{L}^5(\mathbb{L}^2-1)$. Summing up these contributions gives the claimed motive.
\end{proof}

\begin{proposition} With notations as above, the Brauer-Severi scheme of $\mathbb{T}_{3,2}^W(0)$ has a decomposition
\[
\mathbf{BS}_{3,2}^W(0) = \mathbf{S_1} \sqcup \mathbf{S_2} \sqcup \mathbf{S_3} \]
where the schemes $\mathbf{S_i}$ have (equivariant) motives
\[
\begin{cases}
[ \mathbf{S_1} ] = \mathbb{L}^9 + \mathbb{L}^7 - \mathbb{L}^6 \\
[ \mathbf{S_2} ] = \mathbb{L}^8  \\
[ \mathbf{S_3} ] = \mathbb{L}^7  \\
\end{cases}
\]
Therefore, the Brauer-Severi scheme $\mathbf{BS}^W_{3,2}(0)$ has (equivariant) motive
\[
[ \mathbf{BS}_{3,2}^W(0) ] = \mathbb{L}^9 + \mathbb{L}^8 + 2 \mathbb{L}^7 - \mathbb{L}^6    \]
\end{proposition}

\begin{proof} $\mathbf{S_3}$ : If $v-s \not= 0$ we can eliminate $p$ and obtain a contribution $\mathbb{L}^5(\mathbb{L}^2-\mathbb{L})$. If $v=s$ and $n-r \not= 0$ we can eliminate $t$ and obtain a term $\mathbb{L}^4(\mathbb{L}^2-\mathbb{L})$. Finally, if $v=s$ and $n=r$ we have the identity $n^3=0$ and a contribution $\mathbb{L}^5$.

\vskip 3mm

\noindent
$\mathbf{S_2}$ : If $r-n \not= 0$ we can eliminate $x$ and get a term $\mathbb{L}^6(\mathbb{L}^2-\mathbb{L})$. If $r-n=0$ we get the equation in $\mathbb{A}^8$
\[
-\frac{2}{3}n^3+p(vy+w-z) = 1 \]
If $vy+w-z \not= 0$ we can eliminate $p$ and get a contribution $\mathbb{L}^3(\mathbb{L}^4-\mathbb{L}^3)$. Finally, if $vy+w-z=0$ we get the equation $n^3=0$ and hence a term $\mathbb{L}^6$.

\vskip 3mm

\noindent
$\mathbf{S_1}$ : If $(w-z)p+rx \not= 0$ we can eliminate $u$ and obtain a term
\[
\mathbb{L}^4(\mathbb{L}^5-[(w-z)p+rx=0]_{\mathbb{A}^5} )= \mathbb{L}^6(\mathbb{L}-1)(\mathbb{L}^2-1) \]
If $(w-z)p+rx=0$ but $(v-s)p-rt \not= 0$ then we can eliminate $y$ and obtain a contribution
\[
\mathbb{L}[(w-z)p+rx=0,(v-s)p-rt \not= 0]_{\mathbb{A}^8} = \mathbb{L}^5(\mathbb{L}-1)(\mathbb{L}^2-1) \]
Now, assume that $(w-z)p+rx=0$ and $(v-s)p-rt=0$. If $r \not= 0$ then we can eliminate $p,t$ as before and substituting them in the defining equation of $\mathbf{S_1}$ we get 
\[
-\frac{1}{3}r^3-rp=0 \]
and we can eliminate $p$ giving a contribution $\mathbb{L}^6(\mathbb{L}-1)$. Finally, if
 $(w-z)p+rx=0$ and $(v-s)p-rt=0$ and $r=0$ we have the system of equations
\[
\begin{cases} (w-z)p = 0 \\ (v-s)p = 0 \\ (z-w)t+(s-v)x = 0 \end{cases} \]
If $p \not= 0$ we get $w-z=0$ and $v-s=0$ giving a contribution $\mathbb{L}^6(\mathbb{L}-1)$. If $p=0$ the only remaining equation is $(z-w)t+(s-v)x=0$ which gives a contribution $\mathbb{L}^5(\mathbb{L}^2+\mathbb{L}-1)$.
Summing up all terms gives the claimed motive.
\end{proof}

Now, we have all the information to compute the second term of the motivic Donaldson-Thomas series. We have
\[
\begin{cases}
 [\mathbf{BS}_{3,2}^W(0)]-[\mathbf{BS}_{3,2}^W(1)] = \mathbb{L}^7+ \tilde{\mathbf{M}} \mathbb{L}^6 + \tilde{\mathbf{M}} \mathbb{L}^5 \\
[ \mathbb{M}^W_{3,1}(0)]-[\mathbb{M}^W_{3,1}(1)] = \tilde{\mathbf{M}} \mathbb{L}^2
\end{cases} 
\]
By Proposition~\ref{case2} this implies that
\[
(\mathbb{L}^2-1) \frac{[ \mathbb{M}^W_{3,2}(0) ] - [ \mathbb{M}^W_{3,2}(1) ]}{[ GL_2 ]} = \mathbb{L}^7 + \tilde{\mathbf{M}} \mathbb{L}^6 + \tilde{\mathbf{M}} \mathbb{L}^5 + \tilde{\mathbf{M}}^2 \frac{\mathbb{L}^6}{(\mathbb{L}-1)} \]
Therefore the virtual motive is equal to
\[
\mathbb{L}^{-4} \frac{[ \mathbb{M}^W_{3,2}(0) ] - [ \mathbb{M}^W_{3,2}(1) ]}{[ GL_2 ]} = \frac{\mathbb{L}^3(\mathbb{L}-1) + \tilde{\mathbf{M}} \mathbb{L}(\mathbb{L}^2-1) + \tilde{\mathbf{M}}^2 \mathbb{L}^2}{(\mathbb{L}^2-1)(\mathbb{L}-1)} \]
which coincides with the conjectured term in \cite[Conjecture 3.3]{Cazz}.

\end{document}